\documentclass[reqno]{amsart}

\usepackage{enumerate}

\usepackage[colorlinks=true]{hyperref}
\hypersetup{urlcolor=blue, linkcolor=blue, citecolor=red}

\numberwithin{equation}{section}

\newtheorem{thm}{Theorem}[section]
\newtheorem{lem}[thm]{Lemma}
\newtheorem{prop}[thm]{Proposition}

\theoremstyle{definition}
\newtheorem{rem}[thm]{Remark}

\newcommand\R{{\mathbb R}}

\newcommand\Tma{T_{\mathrm{max}}}

\newcommand\Supp{{\mathrm{supp}}\, }

\newcommand\Cz{{C_0(\R^N )}}

\newcommand\Comp{{\mathrm{c}}}

\newcommand\Tsem{{\boldsymbol{\mathcal T}}}

\newcommand\ppoints{\leaders\hbox to 1em{\hss . \hss}\hfill}

\newcommand\MScN[1]{\href{http://www.ams.org/mathscinet-getitem?mr=#1}{(#1)}}
\newcommand\DOI[1]{\href{http://dx.doi.org/#1}{(doi: #1)}}
\newcommand\LINK[1]{\href{#1}{(link: #1)}}

\newcommand\DI{u_0 }

\newcommand\GVar {{\Psi }}
\newcommand\Cgn {{K}}

\title{Finite-time blowup for a complex Ginzburg-Landau equation}

\author[Thierry Cazenave, Fl\'avio Dickstein and Fred B.~Weissler]{}



\thanks{Research supported by the ``Brazilian-French Network in Mathematics"}
\thanks{Fl\'avio Dickstein  was partially supported by CNPq (Brasil) and by a ``Research in Paris" grant from the City of Paris.}
\thanks{Fred B. Weissler benefited from a sabbatical leave (CRCT) from the University of Paris 13.}

\begin{document}
\maketitle

\centerline{\scshape Thierry Cazenave}
\medskip
{\footnotesize
 \centerline{Universit\'e Pierre et Marie Curie \& CNRS}
 \centerline{Laboratoire Jacques-Louis Lions}
   \centerline{B.C.~187, 4 place Jussieu}
   \centerline{75252 Paris Cedex 05, France}
   \centerline{email address: {\href{mailto:thierry.cazenave@upmc.fr}{\tt thierry.cazenave@upmc.fr}}}
}

\medskip

\centerline{\scshape Fl\'avio Dickstein}
\medskip
{\footnotesize
 \centerline{Instituto de Matem\'atica}
 \centerline{Universidade Federal do Rio de Janeiro}
   \centerline{ Caixa Postal 68530}
   \centerline{21944--970 Rio de Janeiro, R.J., Brazil}
   \centerline{email address: {\href{mailto:flavio@labma.ufrj.br}{\tt flavio@labma.ufrj.br}}}
}

\medskip

\centerline{\scshape Fred B.~Weissler}
\medskip
{\footnotesize
 \centerline{Universit\'e Paris 13,  Sorbonne Paris Cit\'e}
 \centerline{CNRS UMR 7539 LAGA}
   \centerline{99 Avenue J.-B. Cl\'e\-ment}
   \centerline{F-93430 Villetaneuse, France}
   \centerline{email address: {\href{mailto:weissler@math.univ-paris13.fr}{\tt weissler@math.univ-paris13.fr}}}
}

\begin{abstract}
We prove that negative energy solutions of the complex Ginzburg-Landau equation $e^{-i\theta } u_t = \Delta u+  |u|^\alpha u$ blow up in finite time, where $\alpha >0$ and $-\pi /2<\theta <\pi /2$.
For a fixed initial value $u(0)$, we obtain estimates of the blow-up time 
$T_{\mathrm{max}}^\theta $ as $\theta  \to \pm \pi /2 $. It turns out that
$T_{\mathrm{max}}^\theta $ stays bounded (respectively, goes to infinity) as $\theta  \to \pm \pi /2 $
 in the case where the solution of the limiting nonlinear Schr\"o\-din\-ger equation blows up in finite time (respectively, is global). 

\end{abstract}

\bigskip 

\section{Introduction} \label{Intro} 

This paper is concerned with the existence of solutions which blow up in finite time of the Cauchy problem
\begin{equation} \label{GL} \tag{GL}
\begin{cases} 
e^{- i\theta } u_t = \Delta u +  |u|^\alpha u, \\
u(0)= \DI,
\end{cases} 
\end{equation} 
in $\R^N $, where $\alpha >0$ and
\begin{equation*} 
-\frac {\pi } {2} \le \theta \le \frac {\pi } {2}.
\end{equation*} 
More precisely, we seek conditions on the initial value $\DI $ which guarantee that the resulting solution is non-global. In addition, we wish to obtain estimates on the blow-up time, for a given initial value $\DI$, as a function of $\theta $.

Equation~\eqref{GL} with  $\theta =0$ reduces to the well known   nonlinear heat equation $u_t-\Delta u= |u|^\alpha u$. For $\theta = \pm \pi /2$, equation~\eqref{GL} becomes the equally well known nonlinear  Schr\"o\-din\-ger equation $ \pm iu_t+\Delta u+  |u|^\alpha u=0$.
Thus we see that~\eqref{GL} is ``intermediate" between the nonlinear heat and Schr\"o\-din\-ger equations. Our overall objective is to understand finite time blowup of solutions of~\eqref{GL} from a unified point of view, for all $-\pi /2\le \theta \le \pi /2$.

The equation~\eqref{GL} is a particular case of the more general complex Ginzburg-Landau equation
\begin{equation} \label{gGL} 
u_t= e^{i\theta }\Delta u + e^{i\gamma } |u|^\alpha u.
\end{equation} 
Equation~\eqref{gGL} has been studied in the context of a wide variety of applications. 
For example, the nonlinear Schr\"o\-din\-ger equation (i.e.~\eqref{gGL} with $\theta =\gamma =\pm \pi /2$) is an important model in nonlinear optics and in the study of weakly nonlinear dispersive waves. We refer the reader to the monograph~\cite{SulemS} which has an extensive discussion of these and other applications. The nonlinear heat equation  (i.e.~\eqref{gGL} with $\theta =\gamma =0$), often with a more general nonlinear term, is also an important model, in particular in biology and chemistry. 
We refer the reader to the monograph~\cite{Fife} for a sampling of such applications.
In the more general case, equation~\eqref{gGL} is used to model such phenomena as  superconductivity, chemical turbulence and various types of fluid flows. 
See~\cite{DoeringGHN} and the references cited therein. 
A key feature associated to the phenomena modeled by~\eqref{gGL} is the development of singularities. Solutions of ~\eqref{gGL} may be global in time or may cease to exist at some finite (blow-up) time.  The existence of  blowing-up solutions  may be interpreted as the appearance of instabilities in the various applications of~\eqref{gGL}.  
 
Local and global existence of solutions of~\eqref{gGL}, on both $\R^N $ and a domain $\Omega \subset \R^N $, 
are known under various boundary conditions and assumptions on the parameters, see e.g.~\cite{DoeringGL, GinibreVu, GinibreVd, MischaikowM, OkazawaYu, OkazawaYd, OkazawaYt, OkazawaYq}.
On the other hand, there are relatively few results concerning the existence of  solutions of~\eqref{gGL} for which finite-time blowup occurs.
In~\cite{Zaag}, blowing-up solutions for the equation~\eqref{gGL} on $\R^N $ are proved to exist, when the equation is ``close" to the nonlinear heat equation $u_t =\Delta u+  |u|^\alpha u$, i.e. when $\theta =0$ and $ |\gamma |$ is small.
A  result in the same spirit is obtained in~\cite{Rottschafer} when the equation is 
``close" to the nonlinear Schr\"o\-din\-ger  equation $i u_t + \Delta u+  |u|^\alpha u=0$. 
The result in~\cite{Zaag} was significantly extended in~\cite{MasmoudiZ}, where the authors give  a rigorous justification of the numerical and formal arguments of~\cite{PlechacS, PoppSKK}.
More precisely, they consider the equation~\eqref{gGL}  on $\R^N $ with $-\pi /2 <\theta ,\gamma <\pi /2$ and  prove the existence of blowing-up solutions when
$\tan^2 \gamma + (\alpha +2) \tan \gamma \tan \theta <\alpha +1$.
Note also that, under certain assumptions on the parameters, blowup for an equation similar to~\eqref{gGL}  on a bounded domain with Dirichlet or periodic boundary conditions, but with the nonlinearity  $ |u|^{\alpha +1}$ instead of $ |u|^\alpha u$  is proved to occur in~\cite{Nasibovu, Nasibovd, OzawaY}.

The equation~\eqref{GL} has certain features not shared by the more general equation~\eqref{gGL}. First of all, stationary solutions of~\eqref{GL} satisfy the same elliptic equation $\Delta u+  |u|^\alpha u=0$, independent of the parameter $\theta $. Furthermore, and more significant for the present article, it turns out that its solutions satisfy energy identities similar to those satisfied by the solutions of the nonlinear heat and Schr\"o\-din\-ger equations. See Proposition~\ref{eGGLu} below.
Recall the energy functional is defined by
\begin{equation}  \label{fGGLt} 
E(w)= \frac {1} {2}\int _{\R^N }  |\nabla w|^2 -\frac {1} {\alpha +2}\int _{\R^N }  |w|^{\alpha +2},
\end{equation} 
for $w\in \Cz \cap H^1(\R^N) $.
This property was exploited in~\cite{SnoussiT}, where the authors apply Levine's argument~\cite{Levine} (see also~\cite{Ball}) and prove finite-time blowup of all negative energy solutions when   $N=1,2$, $\alpha =2$ and $ |\theta |< \pi /4$.
The calculations of~\cite{SnoussiT} can be carried out for more general values of $\alpha $, and the condition  $ |\theta |<\pi /4$ takes the form $\cos^2 \theta >\frac {2} {\alpha +2}$.

Our first main result is that if the initial value $\DI$ has negative energy and $-\pi /2 <\theta <\pi /2$, then the corresponding solution of~\eqref{GL} blows up in finite time. 
We make no assumption on $\alpha >0$. 
We essentially follow the energy method of~\cite{Levine}. The improvement with respect to~\cite{SnoussiT}, where a condition on $\alpha $ and $\theta $ appears, is due to the use of the identity~\eqref{fGGLsbu} below. 

\begin{thm}\label{eGGLt} 
Suppose
\begin{equation} \label{fGGLuA} 
-\frac {\pi } {2} < \theta < \frac {\pi } {2},
\end{equation} 
let $\DI \in \Cz \cap H^1(\R^N) $ and let $u\in C([0, \Tma), \Cz \cap \R^N  )$ be the 
corresponding maximal solution
 of~\eqref{GL}.
If $E(\DI) < 0$, then  $u$ blows up in finite time.
More precisely,
\begin{equation}  \label{feGGLud:s} 
\Tma  \le \frac { \| \DI \| _{ L^2 }^2} {\alpha (\alpha +2) (-E( \DI )) \cos \theta } .
\end{equation} 
\end{thm}  

Of course, $E(\DI )$ in the statement of Theorem~~\ref{eGGLt} refers to the energy functional defined by~\eqref{fGGLt}. 
Theorem~\ref{eGGLt} shows that any solution of~\eqref{GL} with negative initial energy blows up in finite time provided~\eqref{fGGLuA} holds. 
This raises the question of the behavior of the blow-up time as $\theta $ approaches $\pm \pi /2$. 
Indeed, recall that the Cauchy problem for the nonlinear Schr\"o\-din\-ger  equation, i.e. the equation~\eqref{GL}  with $\theta =\pm \pi /2$ is locally well-posed in $H^1 (\R^N ) $ if $\alpha <4/(N-2)$. (See~\cite{GinibreVt, Kato}.) Moreover, if $\alpha <4/N$ then all solutions are global (see~\cite{GinibreVt}), while if $\alpha \ge 4/N$ then some solutions blow up in finite time (see~\cite{Glassey, Zakharov}). More precisely, if the initial value $\DI \in H^1 (\R^N ) $ with negative energy has finite variance (i.e. $\int  |x|^2  |\DI |^2 <\infty $), then the solution blows up in finite time. The same conclusion holds if, instead of assuming that $\DI$ has finite variance, we assume that either $N=1$ and $\alpha =4$, or else $N\ge 2$, $\DI$ is radially symmetric and $\alpha \le 4$, see~\cite{OgawaTu, OgawaTd}. 

Fix an initial value $\DI \in \Cz \cap H^1 (\R^N ) $ such that $E(\DI )<0$ and, given $\theta \in (-\pi /2, \pi /2)$, let $u^\theta $ be the corresponding solution of~\eqref{GL}, so that $u^\theta $ blows up in finite time by Theorem~\ref{eGGLt}. If $\alpha <4/N$, then the solution of~\eqref{GL} for $\theta =\pm \pi /2$ is global, so we may expect that the blow-up time of $u^\theta $ goes to infinity as $\theta \to \pm \pi /2$. This is indeed the case, as the following result shows.

\begin{thm}\label{eGGLud} 
Fix an initial value $\DI \in \Cz \cap H^1 (\R^N  ) $ and, for every $\theta $
satisfying~\eqref{fGGLuA},
let $u^\theta \in C([0, \Tma^\theta ), \Cz \cap H^1  (\R^N ) )$ denote the corresponding maximal solution of~\eqref{GL}. If 
\begin{equation*} 
0<\alpha <\frac {4} {N},
\end{equation*} 
then there exists a constant $c = c( N, \alpha ,  \| \DI\| _{ L^2 }, E( \DI )) >0$ such that
\begin{equation} \label{feGGLud:d} 
 \Tma ^\theta \ge  \frac {c } {\cos \theta } ,
\end{equation} 
 for all $-\frac {\pi } {2}< \theta <\frac {\pi } {2}$.
\end{thm} 

\begin{rem} 
Note that, under the assumptions of Theorem~\ref{eGGLud} and if, in addition, $E( \DI) <0$, there exist $c,C>0$ such that
\begin{equation*} 
\frac {c} {\cos \theta } \le \Tma^\theta \le \frac {C} {\cos \theta } ,
\end{equation*} 
for all $-\pi /2< \theta <\pi /2$. This follows from~\eqref{feGGLud:d} and~\eqref{feGGLud:s}.  
\end{rem} 

Global existence for the nonlinear Schr\"o\-din\-ger equation with $\alpha <4/N$ follows from the conservation of charge and energy and Gagliardo-Nirenberg's inequality. Similarly, Theorem~\ref{eGGLud} follows from energy identities and Gagliardo-Nirenberg's inequality.

\begin{rem} 
Theorems~\ref{eGGLt} and~\ref{eGGLud} are equally valid, with essentially the same proofs, for solutions of~\eqref{GL} on a smooth domain $\Omega \subset \R^N $ with Dirichlet boundary conditions. Moreover, in the case of a bounded domain Ball's proof of finite time blowup~\cite{Ball}
works equally well for~\eqref{GL} with $-\pi /2<\theta <\pi /2$, using the energy identities in Section~\ref{Cauchy}.
\end{rem} 

As observed above, if $\alpha \ge 4/N$ then negative energy, finite variance  solutions of the nonlinear Schr\"o\-din\-ger equation blow up in finite time. Thus we may expect that the blow-up time of $u^\theta $ remains bounded as $\theta \to \pm \pi /2$. 
We have the following result. 

\begin{thm}\label{eGGLhbu} 
Suppose
\begin{equation} \label{feGGLh:u} 
N\ge 2,\quad 
\frac {4} {N} \le \alpha \le 4,
\end{equation} 
and fix a radially symmetric initial value $\DI\in H^1 (\R^N ) \cap \Cz$.
Given any $\theta $ satisfying~\eqref{fGGLuA},
let $u^\theta \in C([0, \Tma^\theta ), \Cz \cap H^1 (\R^N )  )$ denote the corresponding maximal solution of~\eqref{GL}.
If $E(\DI)<0$, then there exists $ \overline{T} <\infty $ such that $\Tma ^\theta \le  \overline{T} $ for all $-\frac {\pi } {2} < \theta <\frac {\pi } {2}$. 
\end{thm} 

Blowup for the equation~\eqref{GL} with $-\frac {\pi } {2} <\theta <\frac {\pi } {2}$
(i.e. Theorem~~\ref{eGGLt})
 is proved by an energy argument. On the other hand, blowup for the nonlinear Schr\"o\-din\-ger equation is proved by a variance argument (or a similar argument for a truncated variance as in~\cite{OgawaTu, OgawaTd}). 
It turns out that for the equation~\eqref{GL} there is also a variance identity (and a truncated variance identity as well), see formulas~\eqref{fVaruqgBu} and~\eqref{fVaruqg} below. By combining the information derived from the truncated variance identity with the energy identities, we are able to establish the uniform estimate of the blow-up time of Theorem~\ref{eGGLhbu}.
We mention that the conditions that $ \DI $ be radially symmetric and that $\alpha \le 4$ are necessary for the crucial estimate in our proof, see Section~\ref{sCmts}.
We do not know if the conclusion of Theorem~\ref{eGGLhbu} is true without these hypotheses.

Note that the assumptions on $\DI$ in Theorem~\ref{eGGLhbu} are precisely those made by Ogawa and Tsutsumi in~\cite{OgawaTu}, where the authors eliminate the finite variance assumption of~\cite{Glassey, Zakharov}. One might expect that, if we were willing to assume that $\DI$ has finite variance, then we would not need the assumptions that $\alpha \le 4$ and that $\DI$ is radially symmetric. In this case, the proof would be based on the variance identity~\eqref{fVaruqgBu} rather than on the truncated variance identity~\eqref{fVaruqg}. Unfortunately, in this case as well, and for apparently different reasons, the same conditions are necessary for the crucial estimate of this other proof. See Section~\ref{Further}.

The rest of this paper is organized as follows. 
In the next section, we recall the basic local well-posedness results for the Cauchy problem~\eqref{GL} and establish the fundamental energy identities.  
Theorems~\ref{eGGLt}, \ref{eGGLud} and~\ref{eGGLhbu} are proved successively in Sections~\ref{Blowup}, \ref{Lower} and~\ref{Upper}.
In Section~\ref{sCmts} we  comment on the obstacles to proving Theorem~\ref{eGGLhbu} under less restrictive hypotheses. In Section~\ref{Further}, we outline the proof which could be given of Theorem~\ref{eGGLhbu} under the additional assumption of finite variance and comment on the related hypotheses. 

\section{The local Cauchy problem: $-\pi /2 < \theta <\pi /2$} \label{Cauchy} 

The linear equation associated with~\eqref{GL} is
 \begin{equation*}
 u_t = e^{i\theta } \Delta u .
 \end{equation*} 
It is well known that the operator $e^{i\theta } \Delta $ with domain $H^2 (\R^N ) $
 generates a semigroup
of contractions $(\Tsem _\theta (t) )_{ t\ge 0 }$ on $L^2 (\R^N ) $. 
Moreover, since~\eqref{fGGLuA} holds,
 the semigroup  $(\Tsem _\theta (t) )_{ t\ge 0 }$ is analytic. 
Indeed, the semigroup $e^{z\Delta }$ is analytic in the half plane $\Re z>0$.
In particular,  $\Tsem_\theta (t)\psi  = G_\theta (t) \star \psi $, where the kernel $G_\theta (t)$ is defined by
\begin{equation*} 
G_\theta (t)  (x)\equiv  (4\pi t e^{i\theta })^{-\frac {N} {2}} e^{- \frac { |x|^2} {4t e^{i\theta }}}.
\end{equation*} 
Since
\begin{equation*} 
 |G_\theta  (t) (x)| = (4\pi t)^{-\frac {N} {2}} e^{- \frac {  |x|^2 \cos \theta } {4 t}},
\end{equation*} 
it follows that
\begin{equation} \label{fInGadbu} 
 \| G_\theta (t)\| _{ L^\sigma  }  = 
 \begin{cases} 
 \displaystyle 
  \sigma ^{-\frac {N} {2\sigma }}  (4\pi  t)^{-\frac {N} {2} (1- \frac {1} {\sigma })}
(\cos \theta ) ^{-\frac {N} {2\sigma }}
 & \text{if }1\le \sigma <\infty  ,\\
(4\pi t)^{-\frac {N} {2}} & \text{if }\sigma =\infty . 
 \end{cases} 
\end{equation} 
We deduce from~\eqref{fInGadbu} and Young's inequality that
\begin{equation} \label{fDispu} 
 \|\Tsem _\theta (t) \psi \| _{ L^r } \le  (\cos \theta )^{-\frac {N} {2} (1- \frac {1} {p}+ \frac {1} {r})}
 t^{-\frac {N} {2} (  \frac {1} {p}- \frac {1} {r})}  \|\psi \| _{ L^p },
\end{equation} 
for $1\le p\le r\le \infty $ and $\theta $ satisfying~\eqref{fGGLuA}.
It follows easily from~\eqref{fDispu} that $(\Tsem _\theta (t) )_{ t\ge 0 }$ is a bounded $C_0$ semigroup on $L^p (\R^N ) $ for $1\le p<\infty $ and on $\Cz$.

It is immediate by a contraction mapping argument that the Cauchy problem~\eqref{GL} is locally well  posed in $\Cz$. Moreover, it is easy to see using the estimates~\eqref{fDispu} that $\Cz \cap H^1 (\R^N ) $ is preserved under the action of~\eqref{GL}. More precisely, we have the following result.

\begin{prop} \label{eGGLzu} 
Suppose~\eqref{fGGLuA}.  
Given any $\DI \in \Cz \cap H^1 (\R^N ) $, there exist $T>0$ and a unique function $u \in  C([0,T], \Cz
\cap H^1 (\R^N ) ) 
\cap C((0,T), H^2 (\R^N ) )\cap C^1((0, T), L^2 (\R^N ) ) $ which satisfies~\eqref{GL} for all $t\in (0,T)$
and such that $u(0)= \DI$. 
Moreover, $u$ can be extended to a maximal interval $[0, \Tma)$, and if $\Tma <\infty $ then $  \|u(t)\| _{ L^\infty  }\to \infty $ as $t \uparrow \Tma$.
\end{prop} 

\begin{rem}  \label{eGGLzd} 
Let $\DI \in \Cz \cap H^1 (\R^N ) $ and let $u$ be the corresponding solution of~\eqref{GL} defined on the maximal interval $[0, \Tma)$, and given by Proposition~\ref{eGGLzu}.
If, in addition, $\alpha <4/N$, then~\eqref{GL} is locally well posed in $L^2 (\R^N ) $ (see~\cite{Weissler}). It is not difficult to show using the estimates~\eqref{fDispu} that the maximal existence times in $\Cz$ and $L^2 (\R^N ) $ are the same;
 and so if $\Tma <\infty $, then $  \|u(t)\| _{ L^2  }\to \infty $ as $t \uparrow \Tma$.
\end{rem} 

We collect below the energy identities that we use in the next sections.

\begin{prop} \label{eGGLu} 
Suppose~\eqref{fGGLuA} and let $\DI \in \Cz \cap H^1 (\R^N ) $. If $u$ is the corresponding  solution of~\eqref{GL} given by Proposition~$\ref{eGGLzu}$ and defined on the maximal interval $[0, \Tma)$, then the following properties hold. 

\begin{enumerate} [{\rm (i)}]

\item \label{eGGLu:u}
Let the energy functional $E$ be defined by~\eqref{fGGLt}. 
It follows that
\begin{equation} \label{fGGLq}
\cos \theta \int _s^t  \int _{\R^N }  |u_t|^2 +  E(u(t))= E(u(s)),
\end{equation} 
for all $0\le s< t< \Tma$. 

\item \label{eGGLu:d}
Set 
\begin{equation}  \label{fGGLc} 
I(w)=  \int _{\R^N }  |\nabla w|^2 - \int _{\R^N }  |w|^{\alpha +2},
\end{equation} 
for $w\in \Cz \cap H^1  (\R^N ) $. It follows that
\begin{equation} \label{fGGLsbu}
\Bigl| \int _{\R^N } u_t  \overline{u}  \Bigr|=  | I(u)|,
\end{equation} 
and 
\begin{equation} \label{fGGLs}
\frac {d} {dt} \int _{\R^N }  |u|^2= - 2\cos \theta   I(u),
\end{equation} 
for all $0< t< \Tma$. 

\end{enumerate} 

\end{prop} 

\begin{proof} 
The identity~\eqref{fGGLq} follows by multiplying the equation~\eqref{GL}  by  $ \overline{u}_t $, integrating by parts on $\R^N  $ and taking the real part. 
Multiplying the equation~\eqref{GL}  by  $e^{i\theta } \overline{u} $ and integrating by parts on $\R^N  $, we obtain
\begin{equation}  \label{fGGLsbd}
\int _{\R^N } u_t  \overline{u}= - e^{i\theta } I(u).
\end{equation} 
Identity~\eqref{fGGLsbu} follows by taking the modulus of both sides of~\eqref{fGGLsbd}, while~\eqref{fGGLs} follows by taking the real part.
\end{proof} 

\begin{rem} 
It follows easily from~\eqref{fGGLs},\eqref{fGGLc} and~\eqref{fGGLq} that
\begin{equation} \label{fGGLn}
\frac {d} {dt} \int _{\R^N}  |u|^2=  \frac {2\alpha } {\alpha + 2} \cos \theta \int _{\R^N}    |u|^{\alpha +2} + 4\cos^2 \theta \int _0^t \int _{\R^N}   |u_t|^2 - 4 \cos \theta E(\DI),
\end{equation}
and
\begin{multline}  \label{fGGLuub}
\frac {d} {dt} \int _{\R^N }  |u|^2=  \alpha \cos \theta  \int _{\R^N }    |\nabla u|^2\\ + 
2(\alpha +2) \cos^2 \theta  \int _0^t \int _{\R^N }  |u_t|^2 - 2(\alpha +2) \cos \theta E(\DI).
\end{multline}   
\end{rem} 

\section{Proof of Theorem~$\ref{eGGLt}$} \label{Blowup} 

We use the argument of~\cite[pp.~185-186]{HarauxW}. 
Note that, by~\eqref{fGGLq},  $E(u(t))\le E(\DI)<0$ for all $0\le t<\Tma$, so that
\begin{equation} \label{fEstIE} 
\begin{split} 
I(u(t)) & = (\alpha +2) E(u(t)) - \frac {\alpha } {2 }\int  _{ \R^N  } |\nabla u(t)|^2 \\ & \le (\alpha +2) E(u(t))
 \le (\alpha +2) E(\DI ) <0 ,
\end{split} 
\end{equation} 
for all $0<t<\Tma$.
Set 
\begin{equation*} 
f(t)=  \|u(t)\| _{ L^2 }^2 ,\quad e(t)= E(u(t)).
\end{equation*} 
We deduce from~\eqref{fGGLq}  that
\begin{equation}  \label{fHWu} 
\frac {de} {dt}   = -\cos \theta   \|u_t\| _{ L^2 }^2 \le 0, 
\end{equation} 
and from~\eqref{fGGLs}  and~\eqref{fEstIE} that  
\begin{equation} \label{fHWd}  
 \frac {df} {dt} = -2\cos \theta  I(u(t)) >0 .
\end{equation} 
It follows from~\eqref{fHWu} and the Cauchy-Schwarz inequality that
\begin{equation} \label{fHWt} 
-f \frac {de} {dt} =f  \cos \theta   \|u_t\| _{ L^2 }^2 = \cos \theta   \|u\| _{ L^2 }^2  \|u_t\| _{ L^2 }^2 
\ge \cos \theta  \Bigl| \int _{\R^N } u_t  \overline{u}  \Bigr|^2. 
\end{equation} 
Using~\eqref{fGGLsbu} and~\eqref{fHWd}, we deduce  that
\begin{equation*} 
\begin{split} 
-f \frac {de} {dt} &  \ge \cos \theta  (I(u(t)))^2= (-I(u(t))) (- \cos \theta I(u(t)))  \\ &
= \frac {1} {2}  (-I(u(t))) \frac {df} {dt}
\ge  \frac {\alpha +2} {2}  (-e) \frac {df} {dt} . 
\end{split} 
\end{equation*} 
This means that
\begin{equation*} 
\frac {d} {dt}(-e f^{-\frac {\alpha +2} {2}})  \ge 0,
\end{equation*} 
so that
\begin{equation} \label{fHWs} 
-e \ge \eta f^{\frac {\alpha +2} {2}},
\end{equation} 
where
\begin{equation} \label{fHWp} 
\eta = ( -E( \DI) ) \| \DI \| _{ L^2 }^{- (\alpha +2)}
\end{equation} 
 It follows from~\eqref{fHWd}, \eqref{fEstIE} and~\eqref{fHWs}  that
\begin{equation*} 
\frac {df} {dt}\ge 2(\alpha +2) (\cos \theta ) (-e) 
 \ge 2 \eta  (\alpha +2) (\cos \theta )f^{\frac {\alpha +2} {2}},
\end{equation*} 
so that
\begin{equation} \label{fHWn} 
\frac {d} {dt}  [  \eta \alpha  (\alpha +2) (\cos \theta ) t +   f^{- \frac {\alpha } {2}} ] \le 0.
\end{equation} 
Integrating~\eqref{fHWn} between $0$ and $t\in (0, \Tma)$, and applying~\eqref{fHWp}, we deduce that
\begin{equation*} 
t\le \frac { \| \DI \| _{ L^2 }^2} {\alpha (\alpha +2) (- E(\DI )) \cos \theta },
\end{equation*} 
for all $0<t<\Tma$. The result follows by letting $t\uparrow \Tma$.

\section{Proof of Theorem~$\ref{eGGLud}$} \label{Lower} 
We first note that by Gagliardo-Nirenberg's inequality there exists $c= c(N )$ such that
\begin{equation} \label{fGNu} 
\int  _{ \R^N  } |u |^{2+ \frac {4} {N}} \le c  \int  _{ \R^N  } |\nabla u|^2  \Bigl( \int  _{ \R^N  } |u|^2 \Bigr)^{\frac {2} {N}}
\end{equation} 
for all $u\in H^1(\R^N)$. Applying H\"older's inequality and~\eqref{fGNu}, we deduce that
\begin{equation} \label{fGNd}
\begin{split} 
\int _{ \R^N  } |u|^{\alpha +2}  & \le  \Bigl( \int  _{ \R^N  } |u|^{2 +\frac {4} {N}} \Bigr)^{ \frac {N\alpha } {4} } \Bigl( \int  _{ \R^N  } |u|^2 \Bigr)^{ \frac {4-N\alpha } {4} } \\ & \le c^{ \frac {N\alpha } {4}}  
 \| \nabla u \| _{ L^2 }^{ \frac {N\alpha } {2} }  \|u\| _{ L^2 }^{  \frac {4- (N-2 )\alpha } {2} }.
\end{split} 
\end{equation} 
We now use Young's inequality
\begin{equation*} 
x y \le \frac {N\alpha } {4} \varepsilon ^{ \frac  {4} {N\alpha }} x^{ \frac  {4} {N\alpha }} + \frac {4-N\alpha } {4} \varepsilon ^{ - \frac {4} {4-N\alpha }} y^{  \frac {4} {4-N\alpha }},
\end{equation*} 
with
\begin{equation*} 
\varepsilon =  \Bigl( \frac {\alpha +2} {N\alpha c} \Bigr)^{ \frac  {N\alpha }{4} },
\end{equation*} 
and we obtain
\begin{equation*} 
\begin{split} 
\frac {1} {\alpha +2}\int _{ \R^N  } |u|^{\alpha +2}  & \le \frac {1} {4}  \|\nabla u\| _{ L^2 }^2 + 
\frac {4-N\alpha } {4(\alpha +2)}  \Bigl( \frac {N\alpha c}  {\alpha +2} \Bigr)^{ \frac  {N\alpha }{4- N\alpha } }  \|u\| _{ L^2 }^{  \frac {2[4- (N-2 )\alpha ]} {4-N\alpha } }  \\
 & \le \frac {1} {4}  \|\nabla u\| _{ L^2 }^2 + 
 ( Nc )^{ \frac  {N\alpha }{4- N\alpha } } \|u\| _{ L^2 }^{  \frac {2[4- (N-2 )\alpha ]} {4-N\alpha } } ,
\end{split} 
\end{equation*} 
so that
\begin{equation}  \label{feGGLud:nbu} 
\frac {1} {\alpha +2}\int _{ \R^N  } |u|^{\alpha +2} \le   \frac {1} {4} \|\nabla u\| _{ L^2 }^2 + 
 \bigl[  ( Nc )^{  N\alpha   }
 \|u\| _{ L^2 }^{  2[4- (N-2 )\alpha  ] }  \bigr]^{  \frac {1 } {4-N\alpha } }.
\end{equation} 
We now prove~\eqref{feGGLud:d}. If $\Tma ^\theta =\infty $, there is nothing to prove. We then assume $\Tma ^\theta <\infty $, so that
\begin{equation} \label{feGGLud:p} 
 \| u^\theta (t) \| _{ L^2 } \uparrow \infty \quad  \text{as}\quad  t\uparrow \Tma ^\theta  ,
\end{equation} 
by Remark~\ref{eGGLzd}. Set 
\begin{equation*} 
S ^\theta = \sup \{ t\in [0,\Tma^\theta );\,  \|u^\theta (s)\| _{ L^2 }^2 \le 2  \| \DI \| _{ L^2 }^2  \text{ for }0\le s\le t  \}.
\end{equation*} 
It follows from~\eqref{feGGLud:p} that $S ^\theta < \Tma ^\theta $ and
\begin{equation}  \label{feGGLud:n} 
 \|u^\theta (S ^\theta )\| _{ L^2 }^2 = 2  \| \DI \| _{ L^2 }^2.
\end{equation} 
Since $E(u^\theta (t))\le E (\DI )$ by~\eqref{fGGLq}  and
\begin{equation} \label{feGGLud:nbt} 
\|u^\theta (t)\| _{ L^2 }^2 \le 2  \| \DI \| _{ L^2 }^2  ,
\end{equation} 
for $0\le t\le S ^\theta $, it follows from~\eqref{feGGLud:nbu} that
\begin{equation} \label{feGGLud:nbd} 
 \| \nabla u^\theta (t)\| _{ L^2 }^2 \le 4E(\DI) +  4   \Cgn  ^{  \frac {1 } {4-N\alpha } } ,
\end{equation} 
where
\begin{equation} \label{fDefCgn} 
\Cgn = ( Nc )^{  N\alpha   }
(2 \| \DI \| _{ L^2 }^2)^{ 4- (N-2 )\alpha   }.
\end{equation} 
Furthermore,  \eqref{feGGLud:nbu}, \eqref{feGGLud:nbt}  and~\eqref{feGGLud:nbd}  imply
\begin{equation*} 
 \|u ^\theta (t) \| _{ L^{\alpha +2} }^{\alpha +2} \le 
(\alpha +2) E(\DI ) +
2(\alpha +2) \Cgn^{  \frac {1 } {4-N\alpha } }  ,
\end{equation*}
so that 
\begin{equation} \label{feGGLud:uz} 
\begin{split} 
 |I(u ^\theta (t)) | &\le  \max  \Bigl\{  \| \nabla u^\theta (t)\| _{ L^2 }^2 ,   \|u^\theta  \| _{ L^{\alpha +2} }^{\alpha +2}  \Bigr\} \\
 &\le (\alpha +4) [E (\DI )]^+ +
  2 (\alpha +2) \Cgn ^{  \frac {1 } {4-N\alpha } }  ,
\end{split} 
\end{equation}
for $0\le t\le S ^\theta $. Applying~\eqref{fGGLs} and~\eqref{feGGLud:uz}, we deduce that
\begin{equation} \label{feGGLud:uu} 
 \|u^\theta ( S^\theta )\| _{ L^2 }^2 \le  \| \DI \| _{ L^2 }^2 + 2  (\cos \theta ) \Bigl[  (\alpha +4) [E (\DI )]^+  +  2 (\alpha +2) \Cgn ^{  \frac {1 } {4-N\alpha } }  \Bigr]  S ^\theta .
\end{equation} 
It now follows from~\eqref{feGGLud:uu} and~\eqref{feGGLud:n} that
\begin{equation} \label{feGGLud:ud} 
S ^\theta  \ge  \frac { \| \DI \| _{ L^2 }^2 } {2 \Bigl[  (\alpha +4) [E (\DI )]^+ \\ +
  2 (\alpha +2) \Cgn ^{  \frac {1 } {4-N\alpha } }
  \Bigr]  \cos \theta  }.
\end{equation} 
Since $\Tma ^\theta \ge S ^\theta $, the result follows from~~\eqref{feGGLud:ud}.  

\begin{rem} 
Suppose $E( \DI) \le 0$. It follows from~\eqref{feGGLud:ud} that
\begin{equation} \label{fLoweru} 
\Tma ^\theta  \ge  \frac { \| \DI \| _{ L^2 }^2 } {4 (\alpha +2) \Cgn ^{  \frac {1 } {4-N\alpha } } \cos \theta  }.
\end{equation} 
For a fixed $\theta $, the right-hand side converges to $0$ very fast as $\alpha \uparrow 4/N$, so the estimate is certainly not optimal with respect to the dependence on $\alpha $.
Compare the estimate from above given in Remark~\ref{eGGLut}.

\end{rem} 

\section{Proof of Theorem~$\ref{eGGLhbu}$} \label{Upper} 

Our proof of Theorem~\ref{eGGLhbu} is modeled on the proof of finite time blowup for the nonlinear Schr\"o\-din\-ger equation (\cite{Zakharov, Glassey, OgawaTu}). 
The basic idea is to estimate $\frac {d^2} {dt^2}\int  \Psi (x) |u|^2$ for an appropriate function $\Psi >0$, in terms of the initial energy $E(\DI )$. If $E(\DI )<0$, this estimate implies that $\int \Psi (x) |u|^2$, becomes negative in finite time, thus showing that the solution cannot be global. 

In the case of~\eqref{GL}, we  have the following generalized variance identity. 

\begin{lem}\label{eGGLuu} 
 Fix a real-valued function $\GVar \in C^\infty (\R^N ) \cap W^{4, \infty } (\R^N ) $. 
Suppose~\eqref{fGGLuA}, let $\DI \in \Cz \cap H^1 (\R^N ) $ and consider 
the  corresponding maximal solution
 $u\in C([0, \Tma), \Cz \cap H^1 (\R^N ) )$  
 of~\eqref{GL}.
 It follows that the map $t\mapsto  \int  _{ \R^N  } \Psi  |u|^2$ belongs to $C^2([0, \Tma))$, 
\begin{multline} \label{fVarug} 
\frac {1} {2}\frac {d} {dt}\int  _{ \R^N  } \GVar  |u|^2 = \cos \theta  \Bigl( -\int  _{ \R^N  } \GVar   |\nabla u|^2 + \int  _{ \R^N  } \GVar   |u|^{\alpha +2} + \frac 12 \int _{ \R^N  } \Delta  \GVar  |u|^2 \Bigr) \\ + \sin \theta  \Im \int _{ \R^N  } \nabla  \GVar  \overline{u} \nabla u,
\end{multline} 
and
\begin{multline} \label{fVaruqg} 
\frac {1} {2}\frac {d^2} {dt^2}\int  _{ \R^N  } \GVar  |u|^2 =
 -\frac {1}{2} \int _{ \R^N  } \Delta^2  \GVar  |u|^2 - \frac {\alpha}{\alpha+2} \int _{ \R^N  } \Delta  \GVar  |u|^{\alpha +2} \\
 +2\Re \int _{ \R^N  } \langle H(\GVar ) \nabla \overline{u},\nabla u\rangle  
+ \cos \theta  \frac {d} {dt}   \int _{ \R^N  } \Bigl\{ - 2  \GVar   |\nabla u|^2+ \frac {\alpha +4} {\alpha +2}   \GVar   |u|^{\alpha +2}
+ \Delta  \GVar   |u|^2 \Bigr\}  
\\ -2 \cos^2 \theta \int _{ \R^N  }  \GVar   |u_t|^2,
\end{multline} 
for all $0\le t<\Tma$, where $H( \GVar )$ is the Hessian matrix $(\partial ^2 _{ ij }\GVar ) _{ i,j }$.
\end{lem} 

\begin{proof} 

Multiplying the equation~\eqref{GL}  by $e^{i\theta }  \GVar (x) \overline{u} $, taking the real part and using the identity
\begin{equation*} 
2 \Re (\nabla \GVar   \overline{u} \nabla u)= \nabla \cdot (\nabla  \GVar  |u|^2) -\Delta  \GVar   |u|^2,
\end{equation*} 
we obtain~\eqref{fVarug}. 
We now differentiate~\eqref{fVarug} with respect to $t$.
We begin with the term in factor of $\sin \theta $ and we note that, using the identity
\begin{equation*} 
\nabla  \GVar   \overline{u} \nabla u_t= \nabla \cdot (\nabla  \GVar  u_t  \overline{u} ) - (\nabla  \GVar  \cdot \nabla  \overline{u}) u_t - \Delta  \GVar  \overline{u} u_t,  
\end{equation*} 
and integration by parts,
\begin{equation*} 
\frac {d} {dt}  \Bigl(  \sin \theta  \Im \int _{ \R^N  } \nabla  \GVar  \overline{u} \nabla u  \Bigr) = - \sin \theta 
 \Bigl(    \Im \int _{ \R^N  } \Delta  \GVar  \overline{u} u_t +2 \Im \int _{ \R^N  } (\nabla  \GVar \cdot \nabla   \overline{u}) u_t \Bigr). 
\end{equation*} 
i.e.
\begin{equation*} 
\frac {d} {dt}  \Bigl(  \sin \theta  \Im \int _{ \R^N  } \nabla  \GVar  \overline{u} \nabla u  \Bigr) =-  \sin \theta 
\Im \int _{ \R^N  } [  \Delta  \GVar    \overline{u}  + 2   \nabla  \GVar  \cdot \nabla   \overline{u}   ] u_t. 
\end{equation*} 
We rewrite this last identity in the form
\begin{multline} \label{fVarqg} 
\frac {d} {dt}  \Bigl(    \sin \theta  \Im \int _{ \R^N  } \nabla  \GVar  \overline{u} \nabla u  \Bigr) =  \cos \theta \Re \int _{ \R^N  } [ 
\Delta  \GVar  \overline{u} +2   \nabla  \GVar  \cdot \nabla  \overline{u} ] u_t \\ -   
\Re \int _{ \R^N  } [  \Delta  \GVar  \overline{u}  + 2   \nabla  \GVar  \cdot \nabla   \overline{u}   ] e^{-i \theta }u_t . 
\end{multline} 
Using~\eqref{GL} and the identities
\begin{gather*} 
\Re (\nabla  \GVar \cdot \nabla  \overline{u})  |u|^\alpha u= \frac {1} {\alpha +2} \nabla \cdot (\nabla  \GVar  |u|^{\alpha +2} ) -
\frac {1} {\alpha +2} \Delta  \GVar  |u|^{\alpha +2},\\
\Re \nabla (\nabla  \GVar \cdot \nabla  \overline{u} )\cdot \nabla u= \frac {1} {2}\nabla \cdot (\nabla  \GVar  |\nabla u|^2) + \Re \langle H( \GVar ) \nabla \overline{u},\nabla u \rangle 
-\frac {1} {2} \Delta  \GVar  |\nabla u|^2,
\end{gather*} 
 we see that
\begin{multline} \label{fVarcg} 
- \Re \int _{ \R^N  } [  \Delta  \GVar     \overline{u}  + 2   \nabla  \GVar  \cdot \nabla   \overline{u}   ] e^{-i \theta }u_t =
-\Re \int _{ \R^N  } [  \Delta  \GVar     \overline{u}  + 2   \nabla  \GVar  \cdot \nabla   \overline{u}   ] (\Delta u+  |u|^\alpha u) \\  = -\frac {1}{2} \int _{ \R^N  } \Delta^2  \GVar  |u|^2 - \frac {\alpha}{\alpha+2} \int _{ \R^N  } \Delta  \GVar  |u|^{\alpha +2} 
  +2\Re \int _{ \R^N  } \langle H( \GVar ) \nabla \overline{u},\nabla u\rangle .
\end{multline}  
We now deduce from~\eqref{fVarqg} and~\eqref{fVarcg}  that
\begin{multline} \label{fVarsg} 
\frac {d} {dt}  \Bigl(    \sin \theta  \Im \int _{ \R^N  } \nabla  \GVar  \overline{u} \nabla u  \Bigr)  = \\
 -\frac {1}{2} \int _{ \R^N  } \Delta^2  \GVar  |u|^2 - \frac {\alpha}{\alpha+2} \int _{ \R^N  } \Delta  \GVar  |u|^{\alpha +2} 
+2\Re \int _{ \R^N  } \langle H( \GVar ) \nabla \overline{u},\nabla u\rangle \\ 
+ \cos \theta \Re \int _{ \R^N  } [ \Delta  \GVar  \overline{u} +2   \nabla  \GVar  \cdot \nabla  \overline{u} ] u_t . 
\end{multline} 
Note that
\begin{equation*}
\begin{split} 
\frac {d} {dt}\int _{ \R^N  }  \GVar   \Bigl( \frac { |\nabla u|^2} {2} - \frac { |u|^{\alpha +2}} {\alpha +2} \Bigr) &=
\Re \int  _{ \R^N  } \GVar  (\nabla  \overline{u} \cdot \nabla u_t -  |u|^\alpha  \overline{u} u_t  ) \\ & =
- \Re \int _{ \R^N  } [ (  \GVar  (\Delta  \overline{u} + |u|^\alpha  \overline{u} ) u_t  ) +  (\nabla  \GVar \cdot \nabla  \overline{u}) u_t \\ & = -\cos \theta \int  _{ \R^N  } \GVar   |u_t|^2 -\Re \int _{ \R^N  } (\nabla  \GVar  \cdot \nabla  \overline{u})u_t,
\end{split} 
\end{equation*} 
so that
\begin{equation}  \label{fVaruug} 
2 \Re \int _{ \R^N  } (\nabla  \GVar  \cdot \nabla  \overline{u})u_t= -2 \cos \theta \int  _{ \R^N  } \GVar   |u_t|^2
- \frac {d} {dt}\int  _{ \R^N  }  \Bigl(  \GVar  |\nabla u|^2 - \frac { 2} {\alpha +2}  \GVar   |u|^{\alpha +2} \Bigr) .
\end{equation} 
Moreover,
\begin{equation} \label{fVaruzg} 
 \Re \int _{ \R^N  } \Delta  \GVar  \overline{u} u_t=  \frac {d} {dt} \frac {1}{2}\int _{ \R^N  } \Delta  \GVar  |u|^2. 
\end{equation} 
We deduce from~\eqref{fVarsg}, \eqref{fVaruug}  and~\eqref{fVaruzg} that
\begin{multline} \label{fVarutg} 
\frac {d} {dt}  \Bigl(    \sin \theta  \Im \int _{ \R^N  } \nabla  \GVar  \overline{u} \nabla u  \Bigr)  = \\
-\frac {1}{2} \int _{ \R^N  } \Delta^2  \GVar  |u|^2 - \frac {\alpha}{\alpha+2} \int _{ \R^N  } \Delta  \GVar  |u|^{\alpha +2} 
 +2\Re \int _{ \R^N  } \langle H( \GVar ) \nabla \overline{u},\nabla u\rangle \\  
+ \cos \theta  \frac {d} {dt}   \int   _{ \R^N  } \Bigl( - \GVar  |\nabla u|^2 + \frac { 2} {\alpha +2}  \GVar   |u|^{\alpha +2} + \frac {1}{2} \Delta  \GVar  |u|^2\Bigr)  
\\ -2 \cos^2 \theta \int _{ \R^N  }  \GVar   |u_t|^2. 
\end{multline} 
Taking now the time-derivative of~\eqref{fVarug} and applying~\eqref{fVarutg}, we obtain~\eqref{fVaruqg}. 
\end{proof}

The next tool we use for the proof of Theorem~\ref{eGGLhbu} is the following estimate. It says that the maximal existence time of a solution $u$ of~\eqref{GL} is controlled, independently of $\theta $, by the maximal time until which $ \| u(t) \| _{ L^2 }$ remains bounded by a (fixed) multiple of $ \| \DI \| _{ L^2 }$.

\begin{lem}\label{eGGLq} 
Suppose~\eqref{fGGLuA}, let $\DI \in \Cz \cap H^1 (\R^N ) $ and consider the 
corresponding maximal solution
 $u\in C([0, \Tma), \Cz \cap H^1 (\R^N )  )$
 of~\eqref{GL}.
 Set
\begin{equation} \label{feGGLq:u} 
\tau = \sup \{ t\in [0,\Tma);\,  \|u(s)\| _{ L^2 }^2 \le K  \| \DI \| _{ L^2 }^2  \text{ for }0\le s\le t  \},
\end{equation} 
where
\begin{equation} \label{feGGLq:d} 
K=  \Bigl[ 1 -  \Bigl( \frac {\alpha +4} {2\alpha +4} \Bigr)^{\frac 12} \Bigr] ^{-1} >1,
\end{equation} 
so that $0\le \tau \le \Tma$. 
 If $E(\DI) \le 0$, then $\Tma \le  \frac {\alpha +4} {\alpha } \tau  $.
\end{lem} 

\begin{proof} 
If $\tau =\Tma$, there is nothing to prove, so we now assume $\tau <\Tma$, so that 
\begin{equation} \label{feGGLq:t} 
  \|u(t)\| _{ L^2 }^2 \le \|u(\tau )\| _{ L^2 }^2= K  \| \DI \| _{ L^2 }^2,\quad 0\le t\le \tau .
\end{equation} 
Since $E( \DI) \le 0$, it follows from~\eqref{fGGLn} that the map $t \mapsto  \|u(t)\| _{ L^2 }$ is nondecreasing on $[0, \Tma)$; and so, using~\eqref{feGGLq:t} 
\begin{equation} \label{feGGLq:q} 
 \|u(t )\| _{ L^2 }^2 \ge  K  \| \DI \| _{ L^2 }^2,\quad \tau \le t<\Tma. 
\end{equation} 
We now use calculations based on Levine~\cite{Levine}.
We deduce from~\eqref{fGGLuub} that
\begin{equation} \label{fGGLud}
\frac {d} {dt} \int _{\R^N }  |u|^2\ge   
2(\alpha +2) \cos^2 \theta  \int _0^t \int _{\R^N }  |u_t|^2.
\end{equation}
Set 
\begin{equation} \label{fGGLut}
h(t)= \int _0^t \int _{\R^N }  |u|^2.
\end{equation} 
It follows from~\eqref{fGGLud} and the Cauchy-Schwarz inequality that
\begin{equation}  \label{fGGLuq}
\begin{split} 
[2(\alpha +2) \cos^2 \theta]^{-1} h h'' & \ge  h \int _0^t \int _{\R^N }  |u_t|^2 
\ge  \Bigl( \int _0^t \int _{\R^N }  |u|\, |u_t| \Bigr)^2 \\ &
\ge  \Bigl( \int _0^t \Bigl| \int _{\R^N } u_t  \overline{u}  \Bigr| \Bigr)^2
\end{split} 
\end{equation} 
Since $I(u(t)) \le (\alpha +2) E(u(t))\le 0$ by~\eqref{fEstIE}, 
identities~\eqref{fGGLsbu} and~\eqref{fGGLs} yield
\begin{equation}  \label{fGGLuqbu}
\Bigl| \int _{\R^N } u_t  \overline{u}  \Bigr| 
=  \frac {1} {2\cos \theta } \frac {d} {dt}\int _{\R^N }  |u|^2
=  \frac {1} {2\cos \theta } h'' (t) .
\end{equation} 
We deduce from~\eqref{fGGLuq} and~\eqref{fGGLuqbu}  that
\begin{equation} \label{fGGLuqbd}
h h'' \ge \frac {\alpha +2} {2} (h'(t) -h'(0))^2.
\end{equation} 
It follows from~\eqref{fGGLuqbd} and~\eqref{feGGLq:q}  that
\begin{equation} \label{feGGLq:c}
h h'' \ge \frac{\alpha +2} {2}  \Bigl( \frac {K-1} {K}  \Bigr)^2 [h'(t)]^2
= \frac{\alpha +4} {4} [h'(t)]^2,
\end{equation} 
for all $\tau \le t< \Tma$. This means that $(h^{-\frac {\alpha } {4} }) '' \le 0$ on $[\tau ,\Tma)$; and so
\begin{equation*} 
h(t)^{-\frac {\alpha } {4}} \le h(\tau )^{-\frac {\alpha } {4}} + (t-\tau ) (h^{-\frac {\alpha } {4}})'(\tau )
= h(\tau )^{-\frac {\alpha } {4}}  \Bigl[ 1- \frac {\alpha } {4}(t-\tau ) h(\tau )^{-1} h'(\tau ) \Bigr],
\end{equation*} 
for $\tau \le t\le \Tma$. Since $h(t)^{-\frac {\alpha } {4}}\ge 0$, we deduce that for every $\tau \le t<\Tma$,
\begin{equation*} 
 \frac {\alpha } {4}(t-\tau ) h(\tau )^{-1} h'(\tau ) \le 1,
\end{equation*} 
i.e. 
\begin{equation} 
(t-\tau )  \|u(\tau )\| _{ L^2 }^2 \le \frac {4} {\alpha } \int _0^\tau   \|u(s)\| _{ L^2 }^2 ds \le  \frac {4} {\alpha } \tau  \|u(\tau )\| _{ L^2 }^2,
\end{equation} 
where we used~\eqref{feGGLq:t} in the last inequality.
Thus $t\le \frac {\alpha +4} {\alpha } \tau $ for all $\tau \le t< \Tma$, which proves the desired inequality.
\end{proof} 

The last ingredient we use in the proof of Theorem~\ref{eGGLhbu} is Lemma~\ref{eOzau} below.
It is an estimate, based on Ogawa and Tsutsumi~\cite{OgawaTu}, which enables us to choose an appropriate function $\Psi $ in Lemma~\ref{eGGLuu}. 
Unfortunately, we have only been able to accomplish this in the radially symmetric case. 
In other words, we are only able to construct a function $\Psi $ for which we can estimate the right-hand side of~\eqref{fVaruqg} for radially symmetric functions $u$.
 
Before stating this result, we rewrite formula~\eqref{fVaruqg} for radially 
 symmetric $\Psi $ and $u$. 
Consider a real-valued function $\GVar \in C^\infty (\R^N ) \cap W^{4, \infty }  (\R^N ) $ as in Lemma~\ref{eGGLuu}, and assume further that $\GVar $ is radially symmetric. 
It follows that 
\begin{equation*} 
 \partial ^2_{j k} \GVar = \frac {\delta  _{ jk }} {r} \GVar '  - \frac {x_j x_k} {r^3}  \GVar ' + \frac {x_j x_k} {r^2}  \GVar '' , 
\end{equation*} 
 so that
 \begin{equation} \label{fYMu} 
\begin{split} 
 \Re \langle H( \GVar ) \nabla \overline{u},\nabla u\rangle & = \frac {\GVar '} {r}  |\nabla u|^2
 -  \Bigl(  \frac {\GVar '} {r^3}- \frac {\GVar ''} {r^2} \Bigr)  |x\cdot \nabla u|^2 \\
 & = \frac {\GVar '} {r}  |\nabla u|^2
 -  \Bigl(  \frac {\GVar '} {r}-  {\GVar ''}  \Bigr)  |\partial _r u|^2 .
\end{split} 
 \end{equation} 
 If, in addition,  $u$ is radially symmetric, then~\eqref{fYMu} yields 
  \begin{equation} \label{fYMubu} 
 \Re \langle H( \GVar ) \nabla \overline{u},\nabla u\rangle =  \GVar '' |u_r|^2 .
 \end{equation} 
It follows from~\eqref{fVaruqg} and~\eqref{fYMubu} that 
if both $u$ and $\GVar$ are radially symmetric, then
\begin{multline} \label{fVaruqs} 
\frac {1} {2}\frac {d^2} {dt^2}\int  _{ \R^N  } \GVar  |u|^2 =2N\alpha  E(u(t)) - (N\alpha -4)  \int _{ \R^N  } | u_r|^2 
 -2  \int _{ \R^N  } (2 - \GVar '' ) |u_r|^2 \\
 + \frac {\alpha  } {\alpha +2}\int _{ \R^N  } (2N-  \Delta  \GVar)  |u|^{\alpha +2}
  -\frac {1}{2} \int _{ \R^N  } \Delta^2  \GVar  |u|^2 \\ 
+ \cos \theta  \frac {d} {dt}   \int _{ \R^N  } \Bigl\{ - 2  \GVar   | \nabla u|^2+ \frac {\alpha +4} {\alpha +2}   \GVar   |u|^{\alpha +2}
+ \Delta  \GVar   |u|^2 \Bigr\}  
\\ -2 \cos^2 \theta \int _{ \R^N  }  \GVar   |u_t|^2. 
\end{multline} 
Since $\Psi (x)$ is radially symmetric, by abuse of notation, we often write $\Psi (x)= \Psi (r)$, where
$r =  |x|$. Using this notation, we have $\Delta \Psi (x) = \Psi ''(r) + \frac {N-1} {r} \Psi '(r)$. We hope the reader will forgive our using both notations in the same formula, as we did in~\eqref{fVaruqs}. 

We now state the needed estimate. Since the proof is an adaptation of arguments in~\cite{OgawaTu} and is somewhat technical, it is given in the appendix~\ref{sProofLemmaeOzau} to this paper. 

\begin{lem} \label{eOzau} 
Suppose $N\ge 2$ and $\alpha \le 4$.
Given any $0<a,A<\infty $, there exists a radially symmetric function $\GVar \in C^\infty (\R^N ) \cap W^{4,\infty } (\R^N ) $, such that $\GVar (x)>0$ for $x\not = 0$ and
\begin{equation} \label{eOzau:u} 
 -2  \int _{ \R^N  } (2 - \GVar '' ) |u_r|^2 \\
 + \frac {\alpha  } {\alpha +2}\int _{ \R^N  } (2N-  \Delta  \GVar)  |u|^{\alpha +2}
  -\frac {1}{2} \int _{ \R^N  } \Delta^2  \GVar  |u|^2\le a,
\end{equation} 
for all radially symmetric $u\in H^1 (\R^N ) $ such that $ \|u\| _{ L^2 }\le A$.
\end{lem} 

\begin{proof}[Proof of Theorem~\ref{eGGLhbu}]
We let $K$ be defined by~\eqref{feGGLq:d} and we set 
\begin{equation} \label{fPRFu} 
\tau ^\theta = \sup \{ t\in [0,\Tma^\theta );\,  \|u^\theta (s)\| _{ L^2 }^2 \le K  \| \DI \| _{ L^2 }^2  \text{ for }0\le s\le t  \},
\end{equation} 
so that
\begin{equation}  \label{fPRFd} 
\sup  _{ 0\le \theta <\frac {\pi } {2} }\sup  _{ 0\le t<\tau ^\theta  }  \|u^\theta (t)\| _{ L^2 }^2 \le K \| \DI \| _{ L^2 }^2.
\end{equation} 
It follows from Lemma~\ref{eGGLq} that
\begin{equation}  \label{fPRFt} 
\Tma ^\theta  \le  \frac {\alpha +4} {\alpha } \tau ^\theta  .
\end{equation} 
We now let $\GVar$ be given by Lemma~\ref{eOzau} with
\begin{equation} \label{fPRFq} 
A= \sqrt K  \| \DI \| _{ L^2 } ,\quad a= - N\alpha E(\DI).
\end{equation} 
Since $E(u^\theta (t))\le E(\DI)$ it follows from~\eqref{fVaruqs}, \eqref{eOzau:u} and \eqref{fPRFq} that
 \begin{multline} \label{fPRFc} 
\frac {1} {2}\frac {d^2} {dt^2}\int _{ \R^N  }  \GVar    |u^\theta |^2 \le  N\alpha  E(\DI )   \\ 
+ \cos \theta  \frac {d} {dt}   \int _{ \R^N  }  \Bigl\{ - 2  \GVar     | \nabla u^\theta  |^2+ \frac {\alpha +4} {\alpha +2}   \GVar     |u^\theta |^{\alpha +2}
+ \Delta  \GVar     |u^\theta |^2 \Bigr\} ,
\end{multline} 
for all $0\le t< \tau ^\theta $. Let
\begin{equation} \label{fPRFcbd} 
B= \int _{ \R^N  } \Bigl\{ - 2  \GVar     | \nabla \DI  |^2+ \frac {\alpha +4} {\alpha +2}   \GVar     |\DI|^{\alpha +2}
+ \Delta  \GVar     |\DI|^2 \Bigr\},
\end{equation} 
and
\begin{multline} \label{fPRFcbu} 
\Gamma _\theta  = \cos \theta  \Bigl( -\int  _{ \R^N  } \GVar   |\nabla \DI|^2 + \int _{ \R^N  }  \GVar   |\DI |^{\alpha +2} + \frac 12 \int _{ \R^N  } \Delta  \GVar  |\DI |^2 \Bigr) \\ + \sin \theta  \Im \int _{ \R^N  } \nabla  \GVar  \overline{\DI } \nabla \DI.
\end{multline} 
Integrating twice the inequality~\eqref{fPRFc} and applying~\eqref{fPRFcbd}-\eqref{fPRFcbu} and~\eqref{fVarug}, we deduce that
 \begin{multline} \label{fPRFs} 
\frac {1} {2} \int  _{ \R^N  } \GVar    |u^\theta |^2 \le \frac {1} {2} \int  _{ \R^N  } \GVar    |\DI |^2 + t  
\Gamma _\theta +   N\alpha  E(\DI )  \frac {t^2} {2} \\ 
+ \cos \theta \int _0^t   \int _{ \R^N  } \Bigl\{ - 2  \GVar     | u^\theta _r |^2+ \frac {\alpha +4} {\alpha +2}   \GVar     |u^\theta |^{\alpha +2}
+ \Delta  \GVar     |u^\theta |^2 \Bigr\}   - B  t  \cos \theta  .
\end{multline} 
On the other hand, it follows from~\eqref{fGGLn} that
\begin{equation*}
\frac {d} {dt} \int _{ \R^N  } |u^\theta |^2 \ge   2\cos \theta  \frac {\alpha } {\alpha + 2}\int _{ \R^N  }  |u^\theta |^{\alpha +2}.
\end{equation*} 
Integrating between $0$ and $t\in (0,\tau ^\theta )$, we obtain
\begin{equation} \label{fPRFh}
2 \cos \theta  \int _0^t \int  _{ \R^N  }  |u^\theta |^{\alpha +2}\le \frac {\alpha +2} {\alpha } [  \|u^\theta (t)\| _{ L^2 }^2 - 
 \|  \DI \| _{ L^2 }^2] \le \frac {\alpha +2} {\alpha } (K-1)  \|  \DI \| _{ L^2 }^2,
\end{equation} 
where we used~\eqref{fPRFd} in the last inequality.
Since $\Psi \in W^{4, \infty }(\R^N )$, it
 now follows from~\eqref{fPRFs}, \eqref{fPRFh} and~\eqref{fPRFd}   that there exists a constant $C$ independent of  $\theta \in (- \pi /2, \pi /2)$ and $t\in (0, \tau ^\theta )$ such that
\begin{equation} \label{fUPTu} 
0\le  C + Ct +   N\alpha  E(\DI )  \frac {t^2} {2} ,
\end{equation} 
for all $0\le t<\tau ^\theta $.
Since $E(\DI) <0$, this implies that there exists $T<\infty $ such that $\tau ^\theta \le T$ for all $-\frac {\pi } {2}< \theta <\frac {\pi } {2}$, and the result follows by applying~\eqref{fPRFt}. 
\end{proof} 

\begin{rem}  \label{eGGLut} 
Suppose $N\ge 2$,  $\alpha <4/N$. Let $\DI \in \Cz \cap H^1 (\R^N ) $ be radially symmetric and satisfy $E( \DI ) <0$. Given $-\pi /2 <\theta <\pi /2$, let $u^\theta $ be the corresponding solution of~\eqref{GL} defined on the maximal interval $[0, \Tma ^\theta )$. It follows in particular from Theorem~\ref{eGGLt} that $u^\theta $ blows up in finite time. 
Using the calculations of the proof of Theorem~\ref{eGGLhbu}, one can improve the estimate~\eqref{feGGLud:s}. 
More precisely, taking into account the term $(4-N\alpha ) \int  _{ \R^N  } |u_r ^\theta |^2 $ in~\eqref{fVaruqs}, instead of~\eqref{fUPTu}, we obtain the inequality 
\begin{equation} \label{fUPTd} 
0\le  C + Ct+ (4-N\alpha ) \int _0^t \int _0^s  \int  _{ \R^N  }  |u_r^\theta |^2 +   N\alpha  E(\DI )  \frac {t^2} {2} ,
\end{equation} 
for all $-\pi /2<\theta <\pi /2$ and $0\le t<\tau ^\theta $.
On the other hand, it follows from~\eqref{fGGLuub} that
\begin{equation*}
\frac {d} {dt} \int _{ \R^N  } |u^\theta |^2 \ge   \alpha \cos \theta   \int _{ \R^N  }  |u_r^\theta |^2.
\end{equation*} 
Integrating between $0$ and $t\in (0,\tau ^\theta )$ and using~\eqref{fPRFd}, we obtain
\begin{equation} \label{fUPTt}
\alpha  \cos \theta  \int _0^t \int  _{ \R^N  }  |u_r^\theta |^2 \le [  \|u^\theta (t)\| _{ L^2 }^2 - 
 \|  \DI \| _{ L^2 }^2] \le (K-1)  \|  \DI \| _{ L^2 }^2.
\end{equation} 
It follows from~\eqref{fUPTd} and~\eqref{fUPTt} that  for some constant $C>0$
\begin{equation} \label{fUPTq} 
0\le  C + C  \Bigl( 1 + \frac {4-N\alpha } {\cos \theta } \Bigr) t +   N\alpha  E(\DI )  \frac {t^2} {2} ,
\end{equation} 
for all $-\pi /2<\theta <\pi /2$ and $0\le t<\tau ^\theta $, which yields the estimate
\begin{equation}  \label{fUPTc} 
\Tma ^\theta  \le C( \DI)  \Bigl( 1+ \frac {4-N\alpha } {\cos \theta } \Bigr).
\end{equation} 
This is interesting,  because we see the dependence in both $\theta $ and $\alpha $. 
It is optimal in $\theta $, but maybe not in $\alpha $. 
(Compare the lower estimate~\eqref{fLoweru}.) 
\end{rem} 

\section{Comments on the hypotheses of Theorem~$\ref{eGGLhbu}$} \label{sCmts} 

As observed above, the assumptions that $\DI $ is radially symmetric and that $\alpha \le 4$ in Theorem~\ref{eGGLhbu} may seem unnatural. 
In this section, we show that both these assumptions are necessary for the method we use. 
Indeed, our proof of Theorem~\ref{eGGLhbu} relies on the identity~\eqref{fVaruqg}.
Assuming that $\GVar\in W^{4, \infty } (\R^N ) \cap C^4(\R^N )$ is radially symmetric, it follows from~\eqref{fVaruqg} and~\eqref{fYMubu} that 
\begin{multline*} 
\frac {1} {2}\frac {d^2} {dt^2}\int  _{ \R^N  } \GVar  |u|^2 = 2N\alpha  E(u(t)) \\ - (N\alpha -4)  \int _{ \R^N  } | \nabla u|^2  + 2 \int  _{ \R^N  } \Bigl( \frac {\Psi '} {r} - \Psi ''\Bigr) ( |\nabla u|^2-  |u_r|^2)
 -2  \int _{ \R^N  } (2 - \GVar '' ) |\nabla u |^2 \\
 + \frac {\alpha  } {\alpha +2}\int _{ \R^N  } (2N-  \Delta  \GVar)  |u|^{\alpha +2}
  -\frac {1}{2} \int _{ \R^N  } \Delta^2  \GVar  |u|^2 
  \\  + \cos \theta  \frac {d} {dt}   \int _{ \R^N  } \Bigl\{ - 2  \GVar   | \nabla u|^2+ \frac {\alpha +4} {\alpha +2}   \GVar   |u|^{\alpha +2}
+ \Delta  \GVar   |u|^2 \Bigr\}   -2 \cos^2 \theta \int _{ \R^N  }  \GVar   |u_t|^2. 
\end{multline*} 
In order to complete our argument, we need at the very least an estimate of the form
\begin{multline} \label{fGenFu} 
- (N\alpha -4)  \int _{ \R^N  } | \nabla u|^2  + 2 \int  _{ \R^N  } \Bigl( \frac {\Psi '} {r} - \Psi ''\Bigr) ( |\nabla u|^2-  |u_r|^2) \\
 -2  \int _{ \R^N  } (2 - \GVar '' ) |\nabla u |^2 
 + \frac {\alpha  } {\alpha +2}\int _{ \R^N  } (2N-  \Delta  \GVar)  |u|^{\alpha +2}
\le F( \|u\| _{ L^2 }), 
\end{multline} 
where $F$ is bounded on bounded sets. Lemma~\ref{eOzau} provides such an estimate for radially symmetric $u$ under the assumption $\alpha \le 4$. 

We claim that if $N\alpha >4$, then there is no radially symmetric $\Psi \in C^4(\R^N ) \cap L^\infty  (\R^N ) $, $\Psi \ge 0$, such that the estimate~\eqref{fGenFu} holds for general $u$. 
To see this, fix $\varphi \in C^\infty _\Comp (\R^N )$, $\varphi \not \equiv 0$ and let 
\begin{equation}\label{fVaruqsbuu} 
u(x)=\lambda^{N/2}\varphi (\lambda (x-x_0)),
\end{equation} 
where $\lambda>0$ and $x_0\in \R^N$. It follows in particular that $\|u\|_{ L^2 }=\|\varphi\|_{ L^2 }$. Given $g\in C(\R^N)$ we have for $\lambda$ large 
\begin{gather}
\int _{ \R^N  } g(x)|\nabla u|^{2} \approx \lambda^{2}g(x_0)\int _{ \R^N  } |\nabla\varphi|^{2}\, dy, \label{fVaruqsbud} \\
\int _{ \R^N  } g(x)|u|^{\alpha+2} \approx \lambda^{N\alpha/2}g(x_0)\int _{ \R^N  } |\varphi|^{\alpha+2}\, dy, \label{fVaruqsbut} \\
\int _{ \R^N  } g(x)|\partial_r u|^{2} \approx \lambda^{2}g(x_0)\int _{ \R^N  } |\partial_r\varphi|^{2}\, dy.\label{fVaruqsbuc}
\end{gather} 
If $N\alpha>4$ and~\eqref{fGenFu} holds, then we deduce from~\eqref{fVaruqsbud}--\eqref{fVaruqsbuc} that $2N-\Delta \GVar(x_0)\le 0$ for all $x_0\in \R^N$, so that $\GVar\not\in L^\infty(\R^N)$. 

We now show that the assumption $\alpha \le 4$ is necessary in order that~\eqref{fGenFu} holds for some $\GVar\in W^{4, \infty } (\R^N ) \cap C^4(\R^N )$ and all radially symmetric $u$. To see this, fix $\varphi \in C^\infty([0,\infty)$ with $\Supp \varphi \subset [1,2 ]$ and $\varphi \not \equiv 0$. For $\lambda >0$ and $r_0>0$ consider 
\begin{equation} \label{fGGLuq:u}
u(x)=\lambda ^{1/2}r_0^{-(N-1)/2} \varphi (\lambda (r-r_0)).
\end{equation} 
Denote by $\omega _N$ the area of the unitary sphere of $\R^N$. 
It follows that for $\lambda \ge 2/r_0$,
\begin{equation} \label{fNL2} 
\begin{split} 
 \|u \| _{ L^2 }^2 
 &=  \omega _N \lambda r_0^{-N+1} \int _0^\infty   |\varphi (\lambda (r- r_0))|^2 r^{N-1}dr \\ 
 &=  \omega _N (\lambda r_0)^{-N+1} \int _1^2  |\varphi (r) |^2 (r+ \lambda r_0)^{N-1}dr \\ 
 &\le   \omega _N (\lambda r_0)^{-N+1} (2+\lambda r_0)^{N-1}  \| \varphi \| _{ L^2(\R) }^2 
 \le 2^{N-1} \omega _N  \| \varphi \| _{ L^2(\R) }^2. 
\end{split} 
\end{equation} 
Given a radially symmetric function $g\in C(\R^N)$ and $r_0>0$ such that $g(r_0)>0$, we have 
as $\lambda \to \infty $
\begin{equation} \label{fGGLuq:t}
\begin{split} 
\int  _{ \R^N } g(x) |u_r|^2 &= 
\omega _N\lambda ^{2} (\lambda r_0 )^{-N+1}\int _1^2 g(\lambda ^{-1} r+r_0) |\varphi '(r)|| ^{2}
 (r+ \lambda r_0) ^{N-1} dr
\\ & \approx \lambda ^{2} \omega _N  g(r_0)  \|\varphi '\| _{ L^2 (\R)}^2,
\end{split} 
\end{equation} 
and, similarly, 
\begin{multline} \label{fGGLuq:q}
\int  _{ \R^N } g(x) |u|^{\alpha +2}  =\omega _N \lambda ^{\frac {\alpha} {2}} r_0 ^{-\frac {(N-1)(\alpha+2)} {2}}\int  _1^2 g(\lambda ^{-1}r+r_0)|\varphi  (r )|^{\alpha+2} (\lambda ^{-1}r+  r_0) ^{N-1}dr \\
 \approx  \lambda ^{\frac {\alpha} {2}}   {\omega _Ng(r_0)} { r_0^{-\frac {(N-1)\alpha } {2}}}  \|\varphi \| _{ L^{\alpha +2 } (\R)}^{\alpha +2}.
\end{multline} 
If $\alpha>4$ and~\eqref{fGenFu} holds, then we deduce from~\eqref{fNL2}--\eqref{fGGLuq:q} that $2N-\Delta \GVar(r_0)\le 0$ for all $r_0 > 0$, so that $\GVar\not\in L^\infty(\R^N)$. 

\section{The variance identity and consequences} \label{Further} 

Another way one might try to dispense with the requirements in Theorem~\ref{eGGLhbu} that $\alpha \le 4$ and that $u_0$ be radially symmetric is to assume that $u_0$ has finite
variance.  
Indeed, finite time blowup of negative energy solutions of the nonlinear Schr\"o\-din\-ger equation, i.e.~\eqref{GL} with $\theta =\pm \pi /2$, was originally proved~\cite{Glassey, Zakharov} for finite variance solutions. No assumption of radial symmetry nor the upper bound  $\alpha \le 4$ was required. These conditions were introduced by Ogawa and Tsutsumi~\cite{OgawaTu}
in their proof of finite time blowup of negative energy solutions (with possibly infinite variance). 
Therefore, it is reasonable to hope that for~\eqref{GL} the additional assumption of finite variance could lead to a proof of finite time blowup without  the assumptions in~\cite{OgawaTu}. 

Consequently, we consider a finite variance solution of~\eqref{GL} which is sufficiently regular 
so that $\GVar=|x|^2$ can be used in formula~\eqref{fVaruqg}. 
This gives
\begin{multline} \label{fVaruqgBu} 
\frac {1} {2}\frac {d^2} {dt^2}\int  _{ \R^N  }  |x|^2  |u|^2 =
2N\alpha E(u(t))
 - (N\alpha - 4) \int _{ \R^N  }  |\nabla u|^2  \\
+ \cos \theta  \frac {d} {dt}   \int _{ \R^N  } \Bigl\{ - 2   |x|^2   |\nabla u|^2+ \frac {\alpha +4} {\alpha +2}    |x|^2    |u|^{\alpha +2}
+ 2N |u|^2 \Bigr\}  
\\ -2 \cos^2 \theta \int _{ \R^N  }   |x|^2   |u_t|^2.
\end{multline} 
These formal calculations can be justified by standard techniques assuming $\DI$ is sufficiently regular, and certainly if $\DI \in C^\infty _\Comp (\R^N )$.
We note right away that the three terms estimated in
Lemma~\ref{eOzau} have disappeared, and so this lemma is no longer needed.
We therefore proceed to outline a proof of the conclusion of Theorem~\ref{eGGLhbu}
based on the formula~\eqref{fVaruqgBu}.
Unfortunately, it will turn out that the conditions that $\alpha \le 4$ and that $\DI $ be radially symmetric will again be required, but for apparently different reasons than in the proof
of Lemma~\ref{eOzau}.

Consider, for simplicity, an initial value $\DI \in C^\infty _\Comp (\R^N )$.  
Suppose~\eqref{fGGLuA} and let $u^\theta $ be the corresponding solution of~\eqref{GL}, defined on the maximal interval $[0, \Tma^\theta )$.

Arguing as in the proof of Theorem~\ref{eGGLhbu} at the end of Section~\ref{Upper}, we obtain that for some $C_1>0$ independent of $\theta$  
\begin{multline} \label{feGGLh:q}
\int _{ \R^N  } |x|^2 |u^\theta |^2 \le \int _{ \R^N  } |x|^2 | \DI |^2
+ C_1 t + N\alpha E(\DI) t^2 \\ 
+ 2\cos \theta \int _0^t \int _{ \R^N  } \Bigl\{ - 2 |x|^2  |\nabla u^\theta |^2+ \frac {\alpha +4} {\alpha +2}  |x|^2  |u^\theta |^{\alpha +2} + 2N   |u^\theta |^2 \Bigr\}  ,
\end{multline} 
for all $0\le t<\Tma^\theta $, see~\eqref{fPRFcbu}, \eqref{fPRFcbd} and \eqref{fPRFs}.
For $K$ defined by~\eqref{feGGLq:d} set $C_2=4N K \| \DI \| _{ L^2 }^2$. If $\tau ^\theta$ is given by~\eqref{fPRFu} then 
\begin{equation} \label{feGGLh:pbu} 
 4N \cos \theta \int _0^t \int _{ \R^N  }   |u^\theta |^2 \le C_2 t,
\end{equation} 
for all $0\le t< \tau ^\theta $, see~\eqref{fPRFd}. Therefore, in order to obtain an inequality analogous to~\eqref{fUPTu}  it remains to estimate the term
\begin{equation} \label{feGGLh:pbd} 
2\cos \theta \int _0^t \int _{ \R^N  } \Bigl\{ - 2 |x|^2  |\nabla u^\theta |^2+ \frac {\alpha +4} {\alpha +2}  |x|^2  |u^\theta |^{\alpha +2}  \Bigr\}.
\end{equation} 
This can be done with the following estimate, similar to some results in~\cite{CaffarelliKN}.

\begin{lem}  \label{eGGLn} 
Suppose $N\ge 2$ and $4/N \le \alpha \le 4$. 
Given any $M>0$, there exists a constant $C$ such that 
\begin{equation}  \label{feGGLn:u} 
\int  |x|^2  |u|^{\alpha +2} \le \int  |x|^2  |\nabla u|^2 + C \int  |u|^{\alpha +2} +C,
\end{equation} 
for all smooth, radially symmetric  $u$ such that $ \|u\| _{ L^2 } \le M$. 
\end{lem} 

\begin{proof} 
We first claim that
\begin{equation} \label{feGGLp:u} 
 \| \,  |\cdot |^N  |u|^2\|_{L^\infty } \le 2  \|u\| _{ L^2}  \|\,  |\cdot | \nabla u\| _{ L^2}.
\end{equation} 
Indeed, considering $u$ as a function of $r>0$, we have
\begin{equation*} 
r^N  |u(r)|^2= - \int _r^\infty \frac {d} {ds}[ s^N  |u(s)|^2 ]= -N \int _r^\infty s^{N-1}  |u(s)|^2 
+ 2 \int _r^\infty s^N \Re ( \overline{u} \partial _r u).
\end{equation*} 
We deduce that
\begin{equation*} 
\begin{split} 
r^N  |u(r)|^2 &\le 2\int _r^\infty s^{N}  |u(s)| \,  |\partial _r u(s)| \\
& \le 2  \Bigl( \int _r^\infty  s^{N-1} |u(s)|^2 \Bigr)^{\frac {1} {2}}
\Bigl( \int _r^\infty  s^{N+1} |\partial _r u(s)|^2 \Bigr)^{\frac {1} {2}}\\
& = 2  \|u\| _{ L^2 (\{  |x|>r \})}  \|\,  |\cdot | \nabla u\| _{ L^2 (\{  |x|>r \})},
\end{split} 
\end{equation*} 
which proves~\eqref{feGGLp:u}. 
It now follows from~\eqref{feGGLp:u} that
\begin{equation*} 
\int  |x|^2  |u|^{\alpha +2} \le  \| \,  |\cdot |^N |u|^2\| _{ L^\infty  }^{\frac {2} {N}} \int  |u|^{\alpha +2-\frac {4} {N}}\le 2^{\frac {2} {N}}  M ^{\frac {2} {N}}   \|\,  |\cdot | \nabla u\| _{ L^2 } ^{\frac {2} {N}}  \int  |u|^{\alpha +2 -\frac {4} {N}}.
\end{equation*} 
Since, by H\"older,
\begin{equation*} 
 \int  |u|^{\alpha +2 -\frac {4} {N}} \le  M^{\frac {8} {N\alpha }}   \Bigl( \int  |u|^{\alpha +2} \Bigr)^{\frac {N\alpha -4} {N\alpha }} ,
\end{equation*}
we deduce that
\begin{equation}  \label{fSppl:d} 
\int  |x|^2  |u|^{\alpha +2} \le 2^{\frac {2} {N}}  M ^{\frac {2\alpha +8} {N\alpha }}   \Bigl( \int  |u|^{\alpha +2} \Bigr)^{\frac {N\alpha -4} {N\alpha }}  
 \|\,  |\cdot | \nabla u\| _{ L^2 } ^{\frac {2} {N}} .
\end{equation} 
Suppose first that $\alpha >4/N$ and fix $0<\eta \le 1$. 
Applying Young's inequality $xy \le \eta ^{-\frac {p} {p'}}\frac {x^p} {p} + \eta \frac {y^{p'}} {p'}$ with $\frac {1} {p}= \frac {N\alpha -4} {N\alpha }$, it follows that
\begin{equation} \label{fSppl:u} 
 2^{- \frac {2} {N}}\int  |x|^2  |u|^{\alpha +2} \le \eta ^{- \frac {4} {N\alpha -4}}   \frac {N\alpha -4} {N\alpha }\int  |u|^{\alpha +2}
+ \eta \frac {4} {N\alpha }  M ^{\frac {\alpha +4} {2}}
 \|\,  |\cdot | \nabla u\| _{ L^2 } ^{\frac {\alpha } {2}} .
\end{equation} 
If $\alpha <4$, then 
we apply again Young's inequality to the last term in the right-hand side of~\eqref{fSppl:u} and we obtain
\begin{multline*} 
 2^{- \frac {2} {N}}\int  |x|^2  |u|^{\alpha +2} \le \eta ^{- \frac {4} {N\alpha -4}}   \frac {N\alpha -4} {N\alpha }\int  |u|^{\alpha +2}   + \frac {\eta} {N}   \|\,  |\cdot | \nabla u\| _{ L^2 } ^2 \\ + 
 \frac {\eta (4-\alpha )} {N\alpha } M^{\frac {2\alpha +8} {4-\alpha }}.
\end{multline*}  
The estimate~\eqref{feGGLn:u} follows  by choosing appropriately $\eta $. 
If $\alpha =4$ (note that $4>4/N$ since $N>1$), then~\eqref{feGGLn:u} follows from~\eqref{fSppl:u} by choosing $\eta$ sufficiently small. 
It remains to consider the case $\alpha =4/N$, in which~\eqref{fSppl:d} becomes
\begin{equation}  \label{fSppl:t} 
\int  |x|^2  |u|^{\alpha +2} \le 2^{\frac {2} {N}}  M ^{\frac {2\alpha +8} {N\alpha }} 
 \|\,  |\cdot | \nabla u\| _{ L^2 } ^{\frac {2} {N}} .
\end{equation} 
Since $N>1$, we may apply Young's inequality to deduce~\eqref{feGGLn:u}.
\end{proof} 

Assuming $N\ge 2$, $4/N \le \alpha \le 4$ and $\DI $ is radially symmetric, 
one can then continue as follows.
Setting $M= \sqrt K  \| \DI \| _{ L^2 } $, we deduce from~\eqref{fPRFd} and Lemma~\ref{eGGLn} that there exists a constant $C_3>0$ such that
\begin{equation}   \label{feGGLh:h} 
\int _{ \R^N  } \Bigl\{ - 2 |x|^2  |\nabla u^\theta |^2+ \frac {\alpha +4} {\alpha +2}  |x|^2  |u^\theta |^{\alpha +2}\Bigr\} \le  C_3 + C_3  \int _{ \R^N  } |u^\theta |^{\alpha +2} ,
\end{equation} 
for all $0\le \theta <\frac {\pi } {2}$ and all $0\le t< \tau ^\theta $.
It follows from~\eqref{feGGLh:q}, \eqref{feGGLh:pbu}  and~\eqref{feGGLh:h}  that 
\begin{multline} \label{feGGLh:n} 
\int _{ \R^N  } |x|^2 |u^\theta |^2 \le \int _{ \R^N  } |x|^2 | \DI |^2
+ (C_1 + C_2+2C_3) t + N\alpha E(\DI) ) t^2 \\ 
+ 2C_3 \cos \theta  \int _0^t \int  _{ \R^N  }  |u^\theta |^{\alpha +2} .
\end{multline} 
Using~\eqref{fPRFh} we see that there exists $C_4$ such that
\begin{equation*} 
\int _{ \R^N  } |x|^2 |u^\theta |^2 \le C_4
+ (C_1 + C_2+2C_3) t + N\alpha E(\DI)  t^2,
\end{equation*} 
for all $-\frac {\pi } {2}\le \theta <\frac {\pi } {2}$ and all $0\le t< \tau ^\theta $.
We then may conclude as in the proof of Theorem~\ref{eGGLhbu}.

Thus we see how to obtain a uniform estimate of $\Tma ^\theta $ by using the variance identity. However, we use Lemma~\ref{eGGLn} and this is why we assume that $\DI$ is radially symmetric and that 
 $N\ge 2$ and $4/N \le \alpha \le 4$. 
 Therefore, we obtain a weaker result than Theorem~\ref{eGGLhbu} (which does not require finite variance). 
 
The obstacle for improving this argument seems to be Lemma~\ref{eGGLn}. 
Unfortunately, both the symmetry assumption and the requirement $\alpha \le 4$ are necessary in 
Lemma~\ref{eGGLn}.

Let us first observe that radial symmetry is essential in Lemma~\ref{eGGLn}.
Indeed, fix $\varphi \in C^\infty _\Comp (\R^N )$, $\varphi \not \equiv 0$ and let $u(x)$ be given by \eqref{fVaruqsbuu}. Taking $g(x) \equiv |x|^2$ in~\eqref{fVaruqsbud} and~\eqref{fVaruqsbut} and $g(x) \equiv 1$ in~\eqref{fVaruqsbut},  we see that~\eqref{feGGLn:u} cannot hold for arbitrarly large $|x_0|$ when $N\alpha>4$. (And not even for $N\alpha=4$, since we may choose $\varphi$ such that 
$\|\varphi\|_{L^{\alpha+2}}^{\alpha+2} \gg \|\nabla \varphi\|_{L^2}^2$.)
 
We next remark that the restriction $\alpha\le 4$ is also essential in Lemma~\ref{eGGLn}. Indeed, let $u$ be defined by \eqref{fGGLuq:u} for some  $\varphi \in C^\infty (\R) $,  $\varphi  \not \equiv 0$ supported in $[1,2]$ and for $\lambda, r_0>0$. 
Applying the first identity in~\eqref{fGGLuq:t} with $g(x) \equiv  |x|^2$ and the 
first identity in~\eqref{fGGLuq:q} with $g(x) \equiv 1$, we deduce that
\begin{gather} 
\int |x|^2 |\nabla u| ^{2}\le \lambda ^{2} 2^{N+1}\omega_N  r_0 ^{2} \|\varphi '\| _{ L^2 (\R)}^2, 
\label{fFinu} \\
\int |u| ^{\alpha+2} \le \lambda ^{\frac {\alpha} {2}} 2^{N-1} \omega_N r_0 ^{-\frac {(N-1)\alpha} {2}} \|\varphi \| _{ L^{\alpha+2}(\R) }^{\alpha+2}, \label{fFind}
\end{gather}   
for all $\lambda \ge 2/r_0$.
Moreover, applying the first identity in~\eqref{fGGLuq:q} with $g(x) \equiv  |x|^2$, we obtain
\begin{equation} \label{fFint}
\int |x|^2 |u| ^{\alpha+2} \ge  \lambda ^{\frac {\alpha} {2}} \omega_N r_0 ^{2-\frac {(N-1)\alpha} {2}} \|\varphi \| _{ L^{\alpha+2} (\R)}^{\alpha+2},
\end{equation}  
for all $\lambda >0$. 
Applying~\eqref{fNL2} and~\eqref{fFinu}--\eqref{fFint}, we see that if~\eqref{feGGLn:u} holds then there is a constant $A>0$ such that
\begin{equation*} 
\lambda ^{\frac {\alpha } {2}} r_0 ^{2-\frac {(N-1)\alpha} {2}} \le A (1 +   \lambda ^{2} r_0 ^{2} +
\lambda ^{\frac {\alpha} {2}} r_0 ^{-\frac {(N-1)\alpha} {2}})
\end{equation*} 
for all $r_0>0$ and $\lambda \ge 2/r_0$. Taking $r_0 = \sqrt{2A} $, we obtain
\begin{equation*} 
\lambda ^{\frac {\alpha } {2}} r_0 ^{-\frac {(N-1)\alpha} {2}} \le 1 +   \lambda ^{2} r_0 ^{2}
\end{equation*} 
for all $\lambda \ge 2/r_0$, which yields $\alpha \le 4$.

\appendix

\section{Proof of Lemma~\ref{eOzau}} \label{sProofLemmaeOzau} 

We follow the method of~\cite{OgawaTu}, and we construct a family $(\GVar _ \varepsilon ) _{ \varepsilon >0 }$ such that, given $a,A$, the estimate~\eqref{eOzau:u} holds with $\GVar = \GVar _\varepsilon $ provided $\varepsilon >0$ is sufficiently small.
Fix  a function $h\in C^\infty  ([0, \infty ))$ such that
\begin{equation} \label{fYTMu} 
h\ge 0,\quad \Supp h\subset [1,2 ], \quad \int _0^\infty h(s)\,ds =1, 
\end{equation} 
 and let
\begin{equation} \label{fYTMubu} 
\zeta (t)= t- \int _0^t (t-s) h(s)\,ds = t- \int _0^t \int _0^s h(\sigma )\,d\sigma ds,
\end{equation} 
for $t\ge 0$. It follows that $\zeta \in C^\infty ([0,\infty )) \cap W^{4,\infty }((0,\infty )$, 
$\zeta '\ge 0$, $ \zeta ''\le 0$, 
$\zeta (t)= t$ for $t\le 1$ and $\zeta (t)=M$ for $t\ge 2$ with $M= \int _0^2 sh(s)\,ds$.
Set
\begin{equation} \label{fLMmu} 
\Phi (x)= \zeta (  |x|^2) .
\end{equation} 
It follows in particular that $\Phi \in C^\infty (\R^N ) \cap W^{4, \infty } (\R^N ) $.
Given any $\varepsilon >0$, set
\begin{equation} \label{fYTMq} 
\GVar _\varepsilon  (x)= \varepsilon ^{-2 } \Phi (\varepsilon x),
\end{equation} 
so that
\begin{equation} \label{fYTMuz} 
  \| \Delta ^2 \GVar _\varepsilon  \| _{ L^\infty  }= \varepsilon ^2  \| \Delta ^2 \Phi \| _{ L^\infty  }. 
\end{equation} 
Next, set
\begin{equation} \label{fLMmd} 
\xi (t)= \sqrt{2( 1-\zeta '(t)) -4t\zeta ''(t)}= \sqrt{2\int _0^t h(s)\,ds +4 th(t)}.
\end{equation} 
It is not difficult to check that $\xi \in C^1([0,\infty )) \cap W^{1, \infty }(0,\infty )$. 
Let
\begin{equation} \label{fLMmt} 
\gamma (r)= \xi (r^2),
\end{equation} 
and, given $\varepsilon >0$, let
\begin{equation} \label{fLMmq} 
\gamma _\varepsilon (r)= \gamma (\varepsilon r) .
\end{equation} 
It easily follows that $\gamma _\varepsilon $ is supported in $[\varepsilon ^{-1}, \infty )$, so that 
\begin{equation} \label{fLMmc} 
\| r^{-(N-1)} \gamma _\varepsilon ' \| _{ L^\infty  } \le \varepsilon ^{N-1} \|  \gamma _\varepsilon ' \| _{ L^\infty  }=  \varepsilon ^N  \|\gamma '\| _{ L^\infty  } ,
\end{equation} 
and
\begin{equation} \label{fLMms} 
  \| r^{-(N-1)} \gamma _\varepsilon u_r\| _{ L^2 }  \le \varepsilon ^{N-1} 
    \|  \gamma _\varepsilon u_r\| _{ L^2 } .
\end{equation} 
Set
\begin{multline} \label{fLMmuz} 
I_\varepsilon (u)= 
 -2  \int _{ \R^N  } (2 - \GVar _\varepsilon '' ) |u_r|^2 \\
 + \frac {\alpha  } {\alpha +2}\int _{ \R^N  } (2N-  \Delta  \GVar_\varepsilon )  |u|^{\alpha +2}
  -\frac {1}{2} \int _{ \R^N  } \Delta^2  \GVar _\varepsilon  |u|^2.
\end{multline} 
Elementary but long calculations using in particular~\eqref{fLMmd} show that
\begin{equation} \label{fDefHeps} 
2- \GVar _\varepsilon ''(x)= \gamma _\varepsilon ( |x|)^2,
\end{equation} 
and
\begin{equation}  \label{fGVlap} 
2N- \Delta \GVar _\varepsilon (x) 
= N  [ \gamma _\varepsilon ( |x|)]^2 + 4(N-1) (\varepsilon  |x|)^2 \zeta ''(\varepsilon ^2 |x|^2) \le 
N  [ \gamma _\varepsilon ( |x|)]^2.
\end{equation} 
We deduce from~\eqref{fLMmuz}, \eqref{fDefHeps},  \eqref{fGVlap} and~\eqref{fYTMuz} that
\begin{equation} \label{fLMmud} 
I_\varepsilon (u) \le -2\int _{ \R^N  } \gamma _\varepsilon ^2  |u_r|^2 + \frac {N\alpha } {\alpha +2}\int _{ \R^N  } \gamma _\varepsilon ^2  |u|^{\alpha +2} + \frac {\varepsilon ^2} {2}  \|\Delta^2 \Phi \| _{ L^\infty  }  \|u\| _{ L^2 }^2. 
\end{equation} 
We next claim that
\begin{equation} \label{fCLu} 
\| \gamma _\varepsilon ^{\frac {1} {2}} u \| _{ L^\infty  }^2 
\le \varepsilon ^N  \|  \gamma  ' \| _{ L^\infty  }
   \|u\| _{ L^2 }^2  + 2 \varepsilon ^{N-1}    \|u\| _{ L^2 }  
    \|  \gamma _\varepsilon  u_r\| _{ L^2 } .
\end{equation} 
Indeed,
\begin{equation} \label{fLMmh} 
\begin{split} 
\gamma _\varepsilon (r)  |u(r)|^2 & = - \int _r^\infty \frac {d} {ds}[\gamma _\varepsilon (s)  |u(s)|^2] 
 \le \int _0^\infty   |\gamma _\varepsilon '|\,  |u|^2 + 2 \int _0^\infty \gamma _\varepsilon  |u|\,  |u_r| \\
& \le   \| r^{-(N-1)} \gamma _\varepsilon ' \| _{ L^\infty  }  \|u\| _{ L^2 }^2 + 2  \|u\| _{ L^2 }  \| r^{-(N-1)} \gamma _\varepsilon u_r\| _{ L^2 }.
\end{split} 
\end{equation} 
(The above calculation is valid for a smooth function $u$ and is easily justified for a general $u$ by density.)
The estimate~\eqref{fCLu}  follows from~\eqref{fLMmh}, \eqref{fLMmc} and~\eqref{fLMms}.   
We now observe that 
\begin{equation} \label{fLMmp} 
\int _{ \R^N  } \gamma _\varepsilon ^2 |u|^{\alpha +2}= \int _{ \R^N  }
\gamma _\varepsilon ^{\frac {4- \alpha } {2}}
[ \gamma _\varepsilon^{\frac {1} {2}}  |u|  ]^\alpha 
 |u|^2\le   \| \gamma  \| _{ L^\infty  }^{\frac {4- \alpha } {2}}  \| \gamma _\varepsilon ^{\frac {1} {2}} u \| _{ L^\infty  }^\alpha   \| u\| _{ L^2 }^2.
\end{equation} 
Applying~\eqref{fCLu} and the inequality $x^{\frac {\alpha } {2}}\le 1+x^2$, we deduce from~\eqref{fLMmp} that there exists a constant $C$ independent of $\varepsilon >0$ and $u$ such that
\begin{equation} \label{fLMmuu} 
 \frac {N\alpha } {\alpha +2} \int _{ \R^N  } \gamma _\varepsilon ^2 |u|^{\alpha +2} \le C   \varepsilon ^{\frac {(N-1) \alpha } {2}}  \|u\| _{ L^2 }^{\frac {\alpha } {2}+ 2}  \Bigl( 
\varepsilon ^{\frac {\alpha } {2}}  \|u\| _{ L^2 }^{\frac {\alpha } {2}} +  1 +  
 \| \gamma _\varepsilon u_r\| _{ L^2 }^2 \Bigr).
\end{equation} 
Estimates~\eqref{fLMmud} and~\eqref{fLMmuu} now yield
\begin{multline} \label{fLMmut} 
I_\varepsilon (u) \le - (2- C   \varepsilon ^{\frac {(N-1) \alpha } {2}}  \|u\| _{ L^2 }^{\frac {\alpha } {2}+ 2} ) \int _{ \R^N  } \gamma _\varepsilon ^2  |u_r|^2 \\ + 
C   \varepsilon ^{\frac {(N-1) \alpha } {2}}  \|u\| _{ L^2 }^{\frac {\alpha } {2}+ 2}  \Bigl( 
\varepsilon ^{\frac {\alpha } {2}}  \|u\| _{ L^2 }^{\frac {\alpha } {2}} +  1 \Bigr)
+ \frac {\varepsilon ^2} {2}  \|\Delta^2 \Phi \| _{ L^\infty  }  \|u\| _{ L^2 }^2. 
\end{multline} 
We now fix $0\le a,A<\infty $ and we first choose $\varepsilon >0$ sufficiently small so that 
$C   \varepsilon ^{\frac {(N-1) \alpha } {2}}  A^{\frac {\alpha } {2}+ 2}\le 2$. It then follows from~\eqref{fLMmut} that if $ \|u\| _{ L^2 }\le A$, then
\begin{equation} \label{fLMmuq} 
I_\varepsilon (u) \le C   \varepsilon ^{\frac {(N-1) \alpha } {2}}  A^{\frac {\alpha } {2}+ 2}  \Bigl( 
\varepsilon ^{\frac {\alpha } {2}}  A^{\frac {\alpha } {2}} +  1 \Bigr)
+ \frac {\varepsilon ^2} {2}  \|\Delta^2 \Phi \| _{ L^\infty  }  A^2.
\end{equation} 
Choosing $\varepsilon >0$ possibly smaller, but depending on $a,A$, we deduce that $I_\varepsilon (u)\le a$ if $ \|u\| _{ L^2 }\le A$.  This completes the proof.

\end{document}